\documentclass{amsart}

\usepackage{amssymb}
\usepackage[all]{xy}

\newtheorem{thm}{Theorem}

\newtheorem{cor}[thm]{Corollary}

\newtheorem{prop}[thm]{Proposition}

\newtheorem{complem}[thm]{Complement}
   
\theoremstyle{definition}
\newtheorem{defn}[thm]{Definition}

\newtheorem{say}[thm]{}

\newtheorem{rem}[thm]{Remark}          

\newtheorem*{ack}{Acknowledgments}      
\newtheorem{notation}[thm]{Notation}   
  
\newtheorem{defn-thm}[thm]{Definition--Theorem}  
\newtheorem{defn-lem}[thm]{Definition--Lemma}  

\newtheorem{comments}[thm]{Comments}

\theoremstyle{remark}

\let \cedilla =\c
\renewcommand{\c}[0]{{\mathbb C}}  

\renewcommand{\o}[0]{{\mathcal O}} 

\newcommand{\dd}[0]{{\mathbb D}}

\newcommand{\q}[0]{{\mathbb Q}}
\newcommand{\map}[0]{\dasharrow}

\newcommand{\spec}[0]{\operatorname{Spec}}

\newcommand{\gal}[0]{\operatorname{Gal}}

\newcommand{\red}[0]{\operatorname{red}}

\newcommand{\rdown}[1]{\lfloor{#1}\rfloor}

\newcommand{\simq}[0]{\sim_{\q}}

\newcommand{\tsum}[0]{\textstyle{\sum}}

\newcommand{\cdiv}[0]{\operatorname{CDiv}}

\newcommand{\defor}[0]{\operatorname{Def}}

\def\loccoh#1.#2.#3.#4.{H^{#1}_{#2}(#3,#4)}

\DeclareMathAlphabet{\mathchanc}{OT1}{pzc}%
                                {m}{it}

\usepackage[all]{xy}\xyoption{dvips}

\begin{document}
\bibliographystyle{amsalpha}


\title{Semi-stable extensions over 1-dimensional bases}
\author{J\'anos Koll\'ar, Johannes Nicaise and Chenyang Xu}

\begin{abstract} 
Given a family of Calabi-Yau varieties over the punctured disc or over the field of Laurent series, we show that, after a finite base change,
the family can be exended across the origin while keeping the canonical class trivial. More generally, we prove similar extension results for families
whose log-canonical class is semi-ample. We use these to show that
the Berkovich and essential skeleta agree for smooth varieties over $\c((t))$
with semi-ample canonical class.
\end{abstract}

\maketitle

Let $C$ be a smooth curve, $C^{\circ}\subset C$ an open subset  and $p^{\circ}:(X^{\circ}, \Delta^{\circ})\to C^{\circ}$
a projective, locally stable family (see Definition \ref{ls.defn}) such that 
$K_{X^{\circ}}+\Delta^{\circ}$ is $p^{\circ}$-nef. 
It is frequently useful to extend it to 
a projective, locally stable family  $p:(X, \Delta)\to C$ such that 
$K_{X}+\Delta$ is $p$-nef. 

This is not possible in general, but  conjecturally 
such an extension exists after pulling back to a suitable finite (ramified)
cover $B\to C$. Currently the latter is known if $C$ is a curve defined  over
a field of characteristic 0,
the fibers of $p^{\circ}$ over $C^{\circ}$ are divisorial log terminal  and $K_{X^{\circ}}+\Delta^{\circ}$ is   $p^{\circ}$-semi-ample.
In increasingly  general forms this has been proved in
\cite{del-mum, ksb, ale-pairs, bchm, MR3032329}.

In many applications, one would like to know  such results  when 
the smooth curve $C$ is replaced by
a 1-dimensional regular scheme or by a (non-compact) Riemann surface.
The cases when  $C=\spec \c[[t]]$ or $C=\dd$ (the complex unit disc) appear especially frequently.

The current proofs rely heavily on the Minimal Model Program, which
is expected to hold in rather general settings. However, complete proofs are known 
only for schemes of finite type,  although the relevant vanishing theorems have been established over  1-dimensional power series rings   as well \cite{MN15}.

It is possible to approximate a general family by an algebraic family to arbitrary order and try to construct the extension by going systematically between
the general family and its approximations. For Calabi--Yau families this was 
worked out in \cite[Sec.4.2]{NX16}, but the general case seems to present technical difficulties. 

 Instead of  approximation, we embed a 1-dimensional formal or analytic family into a higher dimensional algebraic family. Then we solve a special case of the algebraic extension problem over higher dimensional bases and induce the formal or analytic extension form it.

\medskip
From now on all schemes are   over a field of characteristic 0.

\begin{notation}[Global extension problem]\label{basic.setup}
We work in one of the following set-ups.

(Algebraic case)  $C$ is a Noetherian, excellent, 1-dimensional, regular scheme, $C^{\circ}\subset C$ is 
 a dense, open subscheme and $Z:=C\setminus C^{\circ}$.

(Analytic case)   $C$ is a 
(not necessarily compact) Riemann surface and $C^{\circ}\subset C$  is an open subset with finite complement  $Z:=C\setminus C^{\circ}$.

In both cases we also have a  projective morphism $p:(X, \Delta)\to C$
whose restriction 
$p^{\circ}:(X^{\circ}, \Delta^{\circ}):=
\bigl(p^{-1}(C^0), \Delta|_{p^{-1}(C^0)}\bigr)\to C^{\circ}$
is locally stable  with $\q$-factorial, dlt fibers 
 and  $K_{X^{\circ}}+\Delta^{\circ}$ is $p^{\circ}$-semi-ample.

The notions {\it dlt}, $G$-equivariantly dlt,
 {\it qdlt} and $G$-equivariantly  qdlt are discussed in  Definition \ref{ls.defn}; see \cite{km-book, kk-singbook, dkx} for more systematic treatments. 
\end{notation}

The precise extension results are the following.

\begin{thm}[Locally-stable version] \label{main.alg.thm} 
Using the  assumptions of Notation \ref{basic.setup}, 
for every $c_i\in Z$ 
there are natural numbers $m(c_i)$ such that the following holds.

Let $\tau:B\to C$ be  a finite, surjective, Galois  morphism  
such that $m(c_i)$ divides its ramification index over $c_i$ for every $c_i\in Z$.
Then there is 
a projective morphism  $p_B:(X_B, \Delta_B)\to B$ with the following properties.
\begin{enumerate}
\item Over $B^{\circ}:=\tau^{-1}(C^{\circ})$, the morphism $p_B$   is isomorphic to
the pulled-back morphism $(X^{\circ}, \Delta^{\circ})\times_CB^{\circ}\to B^{\circ}$,  
\item $(X_B, \Delta_B+F_{Z,B})$ is dlt, where $F_{Z,B}\subset X_B$ denotes the sum of the fibers of $p_B$ over the points in  $ \tau^{-1}(Z)$,  and
\item $K_{X_B}+\Delta_B+F_{Z,B}$  is $p_B$-semi-ample.
\end{enumerate}
\noindent Furthermore, if $G$ is a finite group acting on $p:(X, \Delta)\to C$
and on $\tau:B\to C$ then we can choose  $p_B:(X_B, \Delta_B)\to B$ such that, in addition, 
\begin{enumerate}\setcounter{enumi}{3}
\item $G_B$ acts on  $p_B:(X_B, \Delta_B)\to B$, where $G_B$ is the natural extension of $G$ by $\gal(B/C)$  and
\item $\bigl(X_B, \Delta_B+F_{Z,B}\bigr)$ is $G_B$-equivariantly dlt.
\end{enumerate}
\end{thm}

Note that (\ref{main.alg.thm}.2) implies that $p_B$ is
locally stable along $F_{Z,B}$, in particular the fibers $F_{Z,B}$ are reduced.
Together with (\ref{main.alg.thm}.1)
this implies that $p_B:(X_B, \Delta_B)\to B$ is locally stable.
(Its fibers are  semi-dlt in the terminology of \cite[Sec.5.4]{kk-singbook}.)
Furthermore, any fiber is a $p_B$-trivial divisor, hence 
(\ref{main.alg.thm}.3) is equivalent to saying that  
 $K_{X_B}+\Delta_B$ is  $p_B$-semi-ample.
For our applications it is important to know that, at least in the local case,
the above model is obtained by an MMP from an snc model.

\begin{complem} \label{main.thm.compl}
Using the  above notation,  assume in addition that $C$ is local with closed point $0$ and $C^{\circ}=C\setminus \{0\}$. Then one can choose
$p_B:(X_B, \Delta_B)\to B$ with the following properties.
\begin{enumerate}
\item There is a projective resolution $(X^1_B, \Delta^1_B)\to (X,\Delta)\times_CB$
such that the pair $(X^1_B, \Delta^1_B+ F^1_{Z,B})$ is  snc.
\item There is a sequence 
$\phi_i:  \bigl(X^i_B, \Delta^i_B\bigr)\map \bigl(X^{i+1}_B, \Delta^{i+1}_B\bigr)$
of $K_{X^i_B}+\Delta^i_B$-negative contractions and flips starting with
$i=1$. 
\item The sequence terminates at  $(X_B, \Delta_B)$ as in
Theorem \ref{main.alg.thm}.
\item The resolution and the sequence  can be chosen to be $\gal(B/C)$-equivariant.
\end{enumerate}
\end{complem}

Taking the quotient of $p_B:(X_B, \Delta_B)\to B$ by $\gal(B/C)$
 we obtain the following  without base change.

\begin{cor}[Base change free version] \label{main.alg.thm.2} 
Using the  assumptions of Notation \ref{basic.setup}, 
 there is  a 
projective morphism  $p_C:(X_C, \Delta_C)\to C$ such that 
\begin{enumerate}
\item over $C^{\circ}$ it is isomorphic to 
$p^{\circ}:(X^{\circ}, \Delta^{\circ})\to C^{\circ}$,  
\item $\bigl(X_C, \Delta_C+\red(F_{Z,C})\bigr)$ is qdlt where $F_{Z,C}\subset X_C$ denotes the sum of the fibers of $p_C$ over the points in $Z$  and
\item $K_{X_C}+\Delta_C+\red(F_{Z,C})$  is $p_C$-semi-ample.
\end{enumerate}
\noindent Furthermore, if $G$ is a finite group acting on  $p:(X, \Delta)\to C$
then we can choose  $p_C:(X_C, \Delta_C)\to C$ such that, in addition, 
\begin{enumerate}\setcounter{enumi}{3}
\item $G$ acts on  $p_C:(X_C, \Delta_C)\to C$ and
\item $\bigl(X_C, \Delta_C+\red(F_{Z,C})\bigr)$ is $G$-equivariantly qdlt.
\end{enumerate}
\end{cor}

Unlike in Theorem \ref{main.alg.thm},   the fibers $F_{Z,C}$ are not reduced
and so  $p_C:(X_C, \Delta_C)\to C$ is not locally stable.
Most likely in (\ref{main.alg.thm.2}.2) one can replace qdlt with dlt.
Note also that 
(\ref{main.alg.thm.2}.3) does not imply that 
 $K_{X_C}+\Delta_C$ is  numerically $p_C$-semiample since $\red(F_{Z,C})$
is not numerically $p_C$-trivial. The following more precise local variant 
follows from Complement \ref{main.thm.compl}.

\begin{complem} \label{main.thm.2.compl}
Using the  above notation,  assume in addition that $C$ is local with closed point $0$ and $C^{\circ}=C\setminus \{0\}$. Then  the following hold.
\begin{enumerate}
\item There is a projective morphism $(X^1_C, \Delta^1_C)\to (X,\Delta)$
such that the pair  $(X^1_C, \Delta^1_C+\red F^1_{Z,C})$ is toroidal and hence qdlt.
(In fact, globally the quotient of an snc pair $(X^1_B, \Delta^1_B+ F^1_{Z,B})$)
\item There is a sequence 
$\phi_i:  \bigl(X^i_C, \Delta^i_C+\red F^i_{Z,C}\bigr)\map \bigl(X^{i+1}_C, \Delta^{i+1}_C+\red F^{i+1}_{Z,C}\bigr)$
of $(K_{X^i_C}+\Delta^i_C+\red F^i_{Z,C})$-negative contractions and flips starting with $i=1$. 
\item The sequence terminates at  $(X_C, \Delta_C)$ as in
Theorem \ref{main.alg.thm.2}.
\end{enumerate}
\end{complem}

\begin{comments} \label{comment.3}
The above results are expected to hold if we only assume that
the fibers over $C^{\circ}$ are proper, slc and with nef log canonical class. 
Our conditions are imposed by current limitations of the MMP. 

If the fibers are not assumed $\q$-factorial, then
the  arguments still work but the resulting  $p_B:(X_B,\Delta_B)\to B$   is isomorphic (over $B^{\circ} $)  to a small $\q$-factorialization of
the pulled-back morphism $(X^{\circ}, \Delta^{\circ})\times_CB^{\circ}\to B^{\circ}$.

By formal gluing theory  (see, for instance, \cite[1.1]{MR1432058})
it is enough to prove the  Complements for the completion (or henselisation) of  the local scheme $(0, C)$.  We need this during the proof in Paragraph \ref{step.3.say}. 

Note that in (\ref{main.thm.compl}.2) and (\ref{main.thm.2.compl}.2) we do not establish that the $\phi_i$  correspond to extremal rays
(rather than faces) but most likely this can be arranged.

The proof of Theorem \ref{main.alg.thm} shows that the expected extension theorems hold over higher dimensional bases, once the global existence theorem of
the moduli space of higher dimensional varieties of general type is written up in \cite{k-modbook}. In writing this note we focus on 1 dimensional bases since they are frequently used in the study of degenerations of
Calabi--Yau varieties and there are complete references available for the background results needed.

One could argue that over a general base scheme local projectivity is more natural than global projectivity and our proof also work if
$p$ is only locally projective.
In this case of course $p_B$ is only locally projective.

It would be more natural to start with a morphism $p^{\circ}:(X^{\circ}, \Delta^{\circ})\to C^{\circ}$ instead of $p:(X, \Delta)\to C$. This is allowed in the algebraic case where a projective family over $C^{\circ}$ extends to a 
projective family over $C$; one can easily induce such an extension from a Hilbert scheme. In the analytic case the difference is, however, substantial.
We still get a morphism from $C^{\circ}$ to a Hilbert scheme but we need to know
that it is meromorphic near $C\setminus C^{\circ}$. This is why we need to start
with a projective (or at least proper and algebraic) family over $C$.
\end{comments}

\begin{ack} This work was started during the workshop
{\it Collapsing
Calabi-Yau Manifolds.} We that the Simons Center for its hospitality and
D.~Abramovich  for helpful
comments and references.
Partial financial support  was provided  to JK   by  the NSF under grant number
 DMS-1362960, to JN  by Starting Grant MOTZETA (project 306610) of the European Research Council.
 and to
CX by the Chinese National Science Fund for Distinguished Young
Scholars.
\end{ack}

\section{MMP for locally stable families over a smooth base}

\begin{defn}\label{ls.defn} 
We follow the usual terminology of \cite{km-book, kk-singbook}
and  assume that the characteristic is 0.
Recall that  `dlt' stands for  divisorial log terminal and `slc'
for semi-log-canonical.

 Let $B$ be a smooth scheme. A morphism  $p:(X, \Delta)\to B$
is called {\it locally stable} if  
$K_{X/B}+\Delta$ is $\q$-Cartier, $p$ is flat  and $\bigl(X_b, \Delta_b)$ is slc for every $b\in B$. The equivalent characterization  (\ref{ss.char.over.samooth.cor}.2) will be especially
 convenient to work with.

(There is a general definition of locally stable morphisms over an arbitrary base $B$, but it is more complicated to state; see \cite{k-modsurv, k-modbook}.)

If $G$ is a finite group acting on a dlt pair  $(X,\Delta)$ then
it called {\it $G$-equivariantly dlt}
or {\it $G$-dlt}  if for every irreducible divisor $D\subset \rdown{\Delta}$
and for every $g\in G$, either $g(D)=D$ or $g(D)\cap D=\emptyset$.
In this case, the quotient $(X/G,\Delta/G)$ need not be dlt. These form the local models of  {\it quotient-dlt} or {\it qdlt} pairs \cite{dkx}.
For practical purposes the following description, proved in
\cite{dkx},  may be the most useful.
\medskip

{\it Definition-Claim \ref{ls.defn}.1.}  A log canonical pair   $(X,\Delta)$ is qdlt iff
there is an open subset $U\subset X$ such that
\begin{enumerate}
\item[(a)] $\bigl(U,  \rdown{\Delta}|_U\bigr)$ is toroidal and
\item[(b)] every  log canonical center of   $(X,\Delta)$ meets $U$.
\end{enumerate}
\medskip

Finally, if $H\subset G$ is a normal subgroup then
$G_1:=G/H$ acts on a qdlt pair $(X/H,\Delta/H)$, these give local models of
{\it $G_1$-equivariantly qdlt}
or {\it $G_1$-qdlt} pairs.

\end{defn}

\begin{prop} Let $(0\in B)$ be a smooth, local scheme of dimension $r$,
$D_1+ \cdots + D_r\subset B$ an snc divisor such that $D_1\cap\cdots \cap D_r=\{0\}$ and $p:(X, \Delta)\to B$ a morphism.
The following are equivalent.
\begin{enumerate}
\item $K_{X/B}+\Delta$ is $\q$-Cartier, $p$ is flat  and $\bigl(X_0, \Delta_0)$ is slc.
\item The pair $\bigl(X, \Delta+p^*D_1+ \cdots + p^*D_r\bigr)$ is slc.
\end{enumerate}
\end{prop}

Proof. We use induction on $r$. Both implications are trivial if $r=0$.

Assume (2). Then $K_X+\Delta+p^*D_1+ \cdots + p^*D_r$ is
 $\q$-Cartier hence so is $K_X+\Delta$.  Set $Y:=p^*D_r$.
Adjunction (cf.\ \cite[4.9]{k-modbook}) shows that
$\bigl(Y, \Delta|_Y+p^*D_1|_Y+ \cdots + p^*D_{r-1}|_Y\bigr)$ is slc, hence
$\bigl(X_0, \Delta_0)$ is slc by induction.

Next we show that $ p^*D_1, \dots, p^*D_r$ form a regular sequence along
$X_0$. First $Y=p^*D_r$ is a log canonical center hence reduced and $S_2$. 
By induction on $r$ the restrictions $p^*D_1|_Y, \dots,  p^*D_{r-1}|_Y$
form a regular sequence. Thus  $ p^*D_1, \dots, p^*D_r$ is a regular sequence along $X_0$, hence $p$ is flat. 

Conversely, assume that (1) holds. 
By induction  $\bigl(Y, \Delta|_Y+p^*D_1|_Y+ \cdots + p^*D_{r-1}|_Y\bigr)$ is slc hence inversion of adjunction (cf.\ \cite[4.9]{k-modbook}) shows that
$\bigl(X, \Delta+p^*D_1+ \cdots + p^*D_r\bigr)$ is slc. \qed

\begin{cor} \label{ss.char.over.samooth.cor}
Let $B$ be a smooth scheme and $p:(X, \Delta)\to B$ a morphism.
The following are equivalent.
\begin{enumerate}
\item $p:(X, \Delta)\to B$ is locally stable.
\item For every snc divisor $D\subset B$, the pair
$(X, \Delta+p^*D)$ is slc. \qed
\end{enumerate}
\end{cor}

\begin{cor}  \label{ss.mmp.over.smooth.cor}
Let $B$ be a smooth scheme and $p:(X, \Delta)\to B$ a projective, locally stable  morphism. 
Let $\phi:(X, \Delta)\map (X^w, \Delta^w)$ be a weak canonical model
over $B$ (cf.\ \cite[3.50]{km-book}). 

Then  $p^w:(X^w, \Delta^w)\to B$ is also locally stable.
\end{cor}

Proof. Let  $D\subset B$ be an snc divisor. By Corollary \ref{ss.char.over.samooth.cor}
$(X, \Delta+p^*D)$ is lc and 
 $p^w:\bigl(X^w, \Delta^w+(p^*D)^w\bigr)\to B$ is also a   weak canonical model
over $B$ by \cite[1.28]{kk-singbook}. Thus $\bigl(X^w, \Delta^w+(p^*D)^w\bigr)$ is also slc.
Next we claim that $ (p^*D)^w=(p^w)^*D$. This is clear away from the
exceptional set of $\phi^{-1}$ which has codimension $\geq 2$ in $X^w$.
Thus  $ (p^*D)^w$ and $(p^w)^*D$ are 2 divisors that agree outside a
codimension $\geq 2$ subset, hence they agree. Now we can use 
Corollary \ref{ss.char.over.samooth.cor} again to conclude that 
$p^w:(X^w, \Delta^w)\to B$ is  locally stable.\qed

\begin{say}[Relative MMP and base change]
Let $B$ be a smooth scheme, $p:(X, \Delta)\to B$ a projective, locally stable  morphism and  $\phi:(X, \Delta)\map (X', \Delta')$ be a 
$K_X+\Delta$-negative contraction or flip.

Let $p:C\to B$ be a  morphism. By base change we get
projective, locally stable  morphisms  $(X_C, \Delta_C)\to C$ and
$(X'_C, \Delta'_C)\to C$. Is the induced morphism
  $\phi_C:(X_C, \Delta_C)\map (X'_C, \Delta'_C)$ a
$K_{X_C}+\Delta_C$-negative contraction or flip?

The answer is clearly yes if $\phi$ is a morphism, but, if the relative Picard number changes then $\phi_C$ could  contract an extremal face. There are, however problems if
$\phi$ is a flip. Let $Z'\subset X'$ denote the exceptional set if $\phi^{-1}$.
It has codimension $\geq 2$, hence there is a closed, nowhere dense subset
$W_1\subset B$ such that $X'_b\cap Z'$ has codimension $\geq 2$
for every $b\in B\setminus W_1$. 
Furthermore,  we claim that $X'_b\cap Z'$ has codimension $\geq 1$
for every $b\in B$. Indeed,  choose a snc divisors $D_1, \dots, D_r$ on $B$ such that
$b$ is  an irreducible component of $D_1\cap \dots\cap D_r$. Then every irreducible component of
$X'_b$ is an lc center of $(X', \Delta'+{p'}^*D_1+ \dots + {p'}^*D_r)$ hence
none of the irreducible components are contained in $Z'$. 

Thus if $p(C)\subset W_1$ then  $Z'_C$ has codimension 1
in $X'_C$ (hence $\phi_C$ is not a flip) but if $p^{-1}(W_1)$ is
nowhere dense in $C$ then $Z'_C$ has codimension $\geq 2$
in $X'_C$ and so $\phi_C$ is  a flip. 
\end{say}

Applying this to a whole sequence of contractions and flips
gives the following.

\begin{prop} \label{mmp.com.base.ch.prop}
Let $B$ be a smooth scheme,  $p:(X, \Delta)\to B$ a projective, locally stable  morphism and 
 $\phi:(X, \Delta)\map (X^w, \Delta^w)$ a composite of
$K_X+\Delta$-negative contractions and flips.
 There is a closed, nowhere dense subset
$W\subset B$ such that the following holds.

Let $C$ be an irreducible, smooth scheme and $p:C\to B$ a morphism
whose image is not contained in $W$. Then, by base change we get
$$
\phi_C:(X_C, \Delta_C)\map (X^w_C, \Delta^w_C)
$$
which is also  a composite of
$K_{X_C}+\Delta_C$-negative contractions and flips. \qed
\end{prop}

\begin{rem}  Applying this to points $b\in B$ we see that 
 the fiber $\bigl((X^w)_b, (\Delta^w)_b\bigr)$ of $p^w$ is 
a weak canonical model of  $\bigl(X_b, \Delta_b\bigr)$
for  $b\in B\setminus W$ but not necessarily if $b\in W$.
It is not easy to find conditions that imply that $W=\emptyset$, not
 even when 
$p^w:\bigl(X^w, \Delta^w\bigr)\to B$ is a minimal or canonical model.
(See \cite[Sec.4]{hmx-bounded} for some cases.)
Thus, although the proof of Corollary \ref{ss.mmp.over.smooth.cor}
is short, the result seems quite surprising to us. 
\end{rem}

Combining Corollary \ref{ss.mmp.over.smooth.cor}, \cite[1.1]{MR3032329} and
\cite[2.9.2]{hmx-bounded} we get the following.

\begin{prop}  \label{H.ss.mmp.over.smooth.cor}
Let $B$ be a smooth, quasi-projective variety, $p:(X, \Delta)\to B$ a projective, locally stable  morphism and $H$  a $p$-ample divisor
 that does not contain any of the log canonical centers. 
Assume that general fibers are dlt and have a minimal model with
semi-ample canonical class. 
Then, for $0<\epsilon\ll 1$,
\begin{enumerate}
\item the canonical model $ \bigl(X^{\rm c}, \Delta^{\rm c}+\epsilon H^{\rm c}\bigr)$ is independent of $\epsilon$,
\item $p^{\rm c}:\bigl(X^{\rm c}, \Delta^{\rm c}+\epsilon H^{\rm c}\bigr)\to B$
is locally stable,
\item $\bigl(X^{\rm c}, \Delta^{\rm c}\bigr) $ is dlt and
\item $K_{X^{\rm c}}+\Delta^{\rm c}$ is $p^{\rm c}$-semiample. \qed
\end{enumerate}
\end{prop}

\section{Locally stable extension of 1-dimensional families}

We start the proof of Theorem \ref{main.alg.thm}   
with the local cases.

\begin{notation}[Local extension problem] We work in one of the following set-ups.

(Algebraic case.) Here $(R, m)$ is a Noetherian, excellent, 1-dimensional regular local ring over
a field of characteristic 0, $(c_0,C):=\spec (R,m)$ and $C^{\circ}=C\setminus \{c_0\}$. 
We further have  a projective morphism   $p:(X, \Delta)\to C$. After restricting to $C^{\circ}$ we get $p^{\circ}:(X^{\circ}, \Delta^{\circ})\to C^{\circ}$.
We assume  that the generic fiber of $p^{\circ}$  is dlt and $K_{X^{\circ}}+\Delta^{\circ}$ is $\q$-Cartier and  semi-ample.

(Analytic case.) Here $(c_0,C)$ is a small complex disc 
with origin $c_0$ and $C^{\circ}=C\setminus \{c_0\}$. 
We further have  a projective morphism  $p:(X, \Delta)\to C$. After restricting to $C^{\circ}$ we get $p^{\circ}:(X^{\circ}, \Delta^{\circ})\to C^{\circ}$.
We assume that the  fibers of $p^{\circ}$ are dlt and $K_{X^{\circ}}+\Delta^{\circ}$ is $\q$-Cartier and $p^{\circ}$-semi-ample.
\end{notation}

The construction of $p_B:(X_B, \Delta_B)\to B$ proceeds in several steps.
We explain the algebraic case in detail and point out the
instances where the analytic case is somewhat different.

\begin{say}[Semi-stable reduction] \label{step.1.say} First we take 
 a log resolution
$\rho:(Y, \Delta_Y)\to (X, \Delta)$ such that 
$\Delta_Y+Y_0$ is an snc divisor where $Y_0$ denotes the central fiber. 
For this the methods of Hironaka are sufficient; more precise references are
\cite{tem} in the algebraic case and \cite[3.44]{k-res} in the analytic case.
These resolutions are also equivariant for any group action.

Furthermore, if $H$ is an ample $\q$-divisor on $X$ then we can assume that
$H_Y:=\rho^*H-E$ is ample on $Y$ where $E$ is a suitable, effecive,  $\rho$-exceptional $\q$-divisor.

Let $n\geq 1$ be  a common multiple of the multiplicities of the
irreducible components of $Y_0$. For any $t\in m_0\setminus m_0^2$
set  $C':=\spec_C \o_C[x]/(x^n-t)\to C$.
By  base change we get
$p':(Y',\Delta'_Y)\to C'$ where $Y'$ denotes the normalization of
$Y\times_CC'$. An easy local computation  shows that  $Y'_0$ is reduced and 
$(Y', \Delta'_Y+Y'_0)$ is toroidal, hence qdlt.
 Thus  $p':(Y',\Delta'_Y)\to C'$  is projective and locally stable with slc central fiber $(Y'_0, \Delta'_0)$ where
$\Delta'_0=\sum a_iD_i$ is a $\q$-linear combination of Cartier divisors $D_i$.

To get the sharpest results, we use \cite[Sec.IV.3]{kkms} which says that,
possibly after a further base change, we can find a
log resolution $(Y'', \Delta''_Y)\to (Y', \Delta'_Y)$ such that
$Y''_0+\Delta''_Y$ is a reduced snc divisor. 
(Note that \cite{kkms} discusses the algebraic case only, but
the proof also works in the analytic setting.)

The pull-back of $H_Y$ gives an ample $\q$-divisor $H''_Y$ on $Y''$
whose restriction to $Y''\setminus Y''_0$ agrees with the pull-back of $H$.
\end{say}

\begin{say}[Universal locally stable deformations]\label{step.2.say} 
Let $(W, \Delta)$ be a projective slc pair where
$\Delta=\sum a_iD_i$ is a $\q$-linear combination of Cartier divisors $D_i$.
For technical reasons we fix an ample divisor class on $W$.
We claim that its locally stable deformations have a universal 
deformation space.

First, $W$ has a universal
deformation  $W^{\rm univ}\to \defor(W)$; see \cite[Chap.III]{har-def}
for a discussion and references. One can choose $\defor(W)$ to be a scheme of finite type and if a finite group acts on $W$ then
one can choose $W^{\rm univ}\to \defor(W)$ to be $G$-equivariant.

By \cite{hacking, k-hh, abr-hass} it has a subspace $\defor^{\rm ls}(W)\subset \defor(W)$
parametrizing those deformations where the relative canonical class
is $\q$-Cartier. Over this there is a universal family 
$\cdiv:=\cdiv\bigl(W^{\rm univ, ls}/\defor^{\rm ls}(W)\bigr)$ parametrizing Cartier divisors  \cite[I.1.13]{rc-book}. For each $D_i$ we take 
a copy of $\cdiv$  and form the fiber product over $\defor^{\rm ls}(W)$.
A neighborhood of the point corresponding to $(W, \Delta)$
is a universal 
deformation space for $(W, \sum a_iD_i)$.  
We denote this space by $\defor^{\rm ls}(W, \Delta)$.
It comes with a universal family
$$
p^{\rm univ}:\bigl(X^{\rm univ}, \Delta^{\rm univ}\bigr)\to 
\defor^{\rm ls}(W, \Delta)
\eqno{(\ref{step.2.say}.1)}
$$
which is projective and locally stable.

(Note that  in general $\defor^{\rm ls}(W, \Delta)$  does depend on the
way we write $\Delta$ in the form $\Delta=\sum a_iD_i$.
For instance if $\Delta=D$ is irreducible then its deformations
are again irreducible. If we write it as $\Delta=\tfrac12 D+\tfrac12 D$
then we allow deformations consisting of 2 divisors  with coefficient $\tfrac12$. 
 This will not be a problem for us.)
\end{say}

\begin{say}[Algebraization of the original family]\label{step.3.say} 
Using the central fiber $(Y''_0, \Delta''_0)$ of the family $p'':(Y'',\Delta''_Y)\to C''$    obtained in Paragraph \ref{step.1.say}, we get
$\defor^{\rm ls}(Y''_0, \Delta''_0)$. After possibly replacing
$C''$ by a suitable \'etale extension (without changing the residue field at the closed point),
there is a  natural moduli map $m''_C:C''\to \defor^{\rm ls}(Y''_0, \Delta''_0)$.
Let
$W''\subset \defor^{\rm ls}(Y''_0, \Delta''_0)$ denote the Zariski closure of the image. As we noted in Paragraph \ref{comment.3}, it is enough to prove
our claims for $C''$.

Next we choose a resolution of singularities  $W\to W''$. By pulling back
the universal family (\ref{step.2.say}.1) to
 $W$, we obtain a  projective, locally stable family
$$
p_W: \bigl(Y_W, \Delta_W\bigr)\to W.
$$
Our conditions were chosen to guarantee that
the MMP (as in  \cite{bchm, MR3032329}) runs and produces a minimal model
$$
p^m_W: \bigl(Y^m_W, \Delta^m_W\bigr)\to W.
\eqno{(\ref{step.3.say}.1)}
$$
If the log canonical class is big on the generic fiber, we also
get a canonical  model
$$
p^{\rm c}_W: \bigl(Y^{\rm c}_W, \Delta^{\rm c}_W\bigr)\to W.
\eqno{(\ref{step.3.say}.2)}
$$
\end{say}

\begin{say}[Local extensions]\label{step.4.say} 
By construction, the family $p'':(Y'',\Delta''_Y)\to C''$
is induced by  base change by the quasi-finite
 moduli map $m''_C:C''\to W$ whose image is Zariski dense.
Thus, as we noted in Proposition \ref{mmp.com.base.ch.prop}, the MMP on $p_W: \bigl(Y_W, \Delta_W\bigr)\to W $ induces an MMP on    $p'':(Y'',\Delta''_Y)\to C''$   which ends with a minimal model.

For many applications this is quite satisfactory.
We usually obtain the original family from some MMP, and
there is no reason to favor one minimal model over another.
However, when we want to glue the local extensions, it is
important to know that we get the correct generic fiber.
This can be achieved in at last 2 ways.

First, in the semi-stable reduction step  (\ref{step.1.say})
we can start with a thrifty log resolution (cf.\ \cite[2.87]{kk-singbook}).
This ensures that when we construct the minimal model, it
will be a small $\q$-factorialization over the generic fiber.
Thus if the generic fiber is $\q$-factorial, we get back the
generic fiber.  This is why we assumed that the fibers are $\q$-factorial.

Second, in Paragraph \ref{step.1.say} we also obtained an ample $\q$-divisor $H''_Y$.
It is not effective, but we can replace it with a $\q$-linearly equivalent
effective ample divisor such that  $\bigl(Y', \Delta'+ H'\bigr)$ is
still dlt.  We can use its new central fiber  
 $\bigl(Y''_0, \Delta''_0+ H''_0\bigr)$ in Paragraph \ref{step.3.say}.
By Proposition \ref{H.ss.mmp.over.smooth.cor} the canonical model 
$$
p^{\rm c}_W: \bigl(Y^{\rm c}_W, \Delta^{\rm c}_W+\epsilon H^{\rm c}\bigr)\to W.
\eqno{(\ref{step.4.say}.1)}
$$
exists, it is independent of $0<\epsilon\ll 1$  and
$K_{Y^{\rm c}_W}+ \Delta^{\rm c}_W$ is  $p^{\rm c}_W$--semi-ample.
 If we use base change from  (\ref{step.4.say}.1),
we get a family
$p'': \bigl(X'', \Delta''\bigr)\to C''$
which has the correct generic fiber.
However,  we can only guarantee that 
 $\bigl(X'', \Delta''+X''_0\bigr) $ is lc, not dlt.
\end{say}

\begin{say}[Global extensions]\label{step.5.say}
We now go back to the original set-up of Theorem \ref{main.alg.thm}. 

We thus have a  projective
morphism $p:(X, \Delta)\to C$ whose restriction
$p^{\circ}:(X^{\circ}, \Delta^{\circ})\to C^{\circ}$ is locally stable 
  with dlt fibers 
 and  $K_{X^{\circ}}+\Delta^{\circ}$ is $p^{\circ}$-semi-ample.
We may assume that $C$ is integral.
For each $c_i\in C\setminus C^{\circ}$ we have obtained natural numbers $m(c_i)$.

Let $\tau:B\to C$ be  a finite, surjective, Galois  morphism  
such that $m(c_i)$ divides its ramification index over $c_i$ for every $c_i\in Z$. 
(It is easy to see that such morphisms exist. For instance, we start
with such local maps $B_i\to (c_i,C)$, let 
 $K\supset k(C)$ denote the Galois closure of the composite
 of  the $k(B_i)$ and $\tau:B\to C$ the normalization of $C$ in $K$.)

The previous local results give extensions of 
$p_B^{\circ}:(X^{\circ}, \Delta^{\circ})\times_CB^{\circ}\to B^{\circ}$
locally over each $c_i$ and these can be glued together
to $p_B:(X_B, \Delta_B)\to B$.

As we discussed in Paragraph \ref{step.4.say}, we can keep track of a
relatively ample divisor and ensure that $p_B$ is projective.
\end{say}

\begin{say}[Equivariant versions]\label{step.6.say}
Assume next that a finite group acts on 
$p:(X,\Delta)\to C$. For each $s\in Z$ let $G_s\subset G$ denote the stabilizer of $s$. The resolution can be done $G_s$-equivariantly.
We then use the normalization of $C_s$ in the splitting field of $x^m-t$
to get  $C''_s\to C_s$ with Galois group $H_s$. 
 The canonical model is automatically $G_s\times H_s$-equivariant.

To achieve globalization, we let $K$ be the 
 Galois closure of the composite
of all of the $k(C''_s)$ over $k(C/G)$. 
Set $H:=\gal\bigl(K/k(C/G)\bigr)$ and 
$H_1:=\gal\bigl(K/k(C)\bigr)$. The actions on $C$ give a natural
isomorphism  $G/G_1\cong H/H_1$. Let $G_B\subset G\times H$ denote the
preimage of the diagonal of $G/G_1\times H/H_1$. Then $G_B$ acts on
$(X, \Delta)\times_CB\to B$ and our constructions are $G_B$-equivariant.
Thus the resulting  $p_B:(X_B, \Delta_B)\to B$ comes with a
$G_B$-action.
\end{say}

\begin{say}[Extension without base change]\label{step.7.say}
We take the above $p_B:(X_B, \Delta_B)\to B$  with its
$G_B$-action  and set
$(X_C, \Delta_C):=(X_B, \Delta_B)/H_1$. Note that the $H_1$-action may leave some irreducible components of the fibers fixed, so
$X_B\to X_C$ can be ramified along some divisors in $F_{Z,B}$.
Thus $K_{X_B}+\Delta_B$ need not be the pull-back of
$K_{X_C}+\Delta_C$ but $K_{X_B}+\Delta_B+F_{Z,B}$ is the pull-back of
$K_{X_C}+\Delta_C+F_{Z,C}$. 

If the $H_1$-action has any fixed points, we can only guarantee
that  $(X_C, \Delta_C)$ has qdlt singularities. 
As explained in \cite[Sec.5]{dkx}, we should be able to pass to
partial resolution with dlt singularities, but this again involves a version of the 
MMP that is not known in our case. 
\end{say}

\section{The essential skeleton of degenerations}

\begin{say}[Berkovich and essential skeleta]
Let $R$ be a complete, discrete valuation ring with residue field $k$ and quotient field $K$. We assume that $k$ has characteristic zero.
  For every $K$-scheme of finite type $X^\circ$ we denote by $X^{\rm an}$ its analytification in the sense of Berkovich.  (See \cite{MR1070709, 2014arXiv1409.5229N} for  introductions to Berkovich analytic spaces.)
 Let $X$ be a regular, flat $R$-scheme of finite type such that the special fiber $X_k$ has simple normal crossings and   $X^\circ\cong X_K$. Then there is a canonical embedding of the dual intersection complex of $X_k$ into $X^{\rm an}$. The image of this embedding is called the {\it Berkovich skeleton} of $X$ and denoted by ${\rm Sk}(X)$.
    If $X$ is proper over $R$, then
      one can deduce from results of \cite{MR1070709, MR2320738} that ${\rm Sk}(X)$ is a strong deformation retract of $X^{\rm an}$; see  \cite[Thm.3.1.3]{NX16} and the  remarks after it.

 Assume next that  $X^\circ$ is a smooth, projective $K$-scheme and has a  semi-ample canonical class $K_{X^{\circ}}$.
 The paper  \cite{MN15} defines the {\em essential skeleton} of ${\rm Sk}(X^{\circ})$. It is  a  topological subspace of $X^{\rm an}$ with a piecewise integral affine structure obtained as  a union of faces of the Berkovich skeleton  ${\rm Sk}(X)$, but  ${\rm Sk}(X^{\circ})$  does not depend on the choice of $X$.
   It is proved in \cite{NX16} that if $X$ is a minimal, dlt model of $X^\circ$ over $R$, then ${\rm Sk}(X^{\circ})$ is canonically homeomorphic to the dual intersection complex of $X_k$ and it is a strong deformation retract of
 $X^{\rm an}$. To guarantee the existence of such a
  minimal, dlt model $X$,  \cite{NX16} assumed that $X^\circ$ is defined over a $k$-curve. We will now remove this assumption using the results from the previous sections.

More generally, let $X$ be a  proper, flat $R$-scheme such that
$X_K$ is smooth and   $(X, \red X_k)$ is qdlt
as in  Definition  \ref{ls.defn}. Let $U\subset X$ be an open subset as in
(\ref{ls.defn}.1)  and  $V\to U$ 
  any projective, toroidal resolution of the pair $(U,\red U_k)$. 
As in  \cite[Prop.37]{dkx} we obtain that
 ${\rm Sk}(V)\subset X^{\rm an}$ is independent of the choice of $V$ and that it is a subdivision of the dual intersection complex of $\red X_k$. We call ${\rm Sk}(X):={\rm Sk}(V)$  the {\it Berkovich skeleton} of $X$. If $(X,\red X_k)$ is toroidal then we can take $U=X$ and $V$ is an snc-model of $X^{\circ}$.
In particular,  ${\rm Sk}(X)$ is a strong deformation retract of $X^{\rm an}$.
\end{say}

\begin{thm} \label{berk=ess.thm}
Let $X^{\circ}$ be a projective smooth $K$-variety with semi-ample canonical class. Let $X$ be a projective, qdlt model
 of $X^\circ$ over $R$ such that $K_X+\red X_k$ is semi-ample. 
\begin{enumerate}
\item  The Berkovich skeleton ${\rm Sk}(X)$ coincides with the essential skeleton ${\rm Sk}(X^{\circ})$. 
\item ${\rm Sk}(X^{\circ})$ is a strong deformation retract of $X^{\rm an}$.
\end{enumerate}
\end{thm}

\begin{proof} Note that a model $X$ as in the statement exists by Corollary \ref{main.alg.thm.2}. 

As in (\ref{ls.defn}.1) let  $U\subset X$ be the maximal open subscheme such that $(U,\red U_k)$ is toroidal. Choose a toroidal, projective resolution
$V\to U$ and extend it to a projective log resolution $h:Y\to X$
of $(X, \red X_k)$.  Thus ${\rm Sk}(X)={\rm Sk}(V)$ by definition. 
Write
$$
K_{Y}+\red Y\simq h^*\bigl(K_{X}+\red X_k\bigr)+\tsum_{E}a_E E,
$$
where the divisors $E$ are $h$-exceptional. Over $U$ we have a toroidal resolution, hence the $h$-exceptional divisors that meet $V$ have discrepancy $-1$ and they are contained in $\red Y$. Thus $h(E)\subset X\setminus U$
whenever $a_E\neq 0$ and hence $a_E>0$ for every such $E$ since
none of the log canonical centers is contained in $X\setminus U$.
A straightforward generalization of the arguments in \cite[3.3.2--3]{NX16} now shows that ${\rm Sk}(X)$ coincides with the essential skeleton ${\rm Sk}(X^{\circ})$, proving (1).

Next we prove (2). 
Since the essential skeleton ${\rm Sk}(X^{\circ})$ depends only on the generic fiber, we can choose   our minimal, qdlt-model as in  Complement \ref{main.thm.2.compl}. That is, 
there is a sequence of models $X^i$  such that  $(X^0, \red X^0_k)$ is toroidal, 
each $\phi^i:X^i\map X^{i+1}$ is a $\bigl(K_{X^i}+\red X^i_k\bigr)$-negative contraction or flip
 and $\bigl(X^m, \red X^m_k\bigr)=(X, \red X_k)$.
 We already noted that ${\rm Sk}(X^0)$ is a strong deformation retract of $X^{\rm an}$, so that it suffices to show that ${\rm Sk}(X)$ is a strong deformation retract of ${\rm Sk}(X^0)$.

 Since  \cite[Prop.25]{dkx} is still true in this setting, we can apply  \cite[Thm.19]{dkx} to conclude that 
 the dual complex  ${\rm D}\bigl(X^{i+1}_k\bigr)$ is obtained from  ${\rm D}\bigl(X^{i}_k\bigr)$ by removing all the faces corresponding to the strata of $Z\subset X^{i}_{k}$ such that $\phi^i$ is not an open embedding at the generic point of $Z$. Arguing as in \cite[Lem.21]{dkx} we conclude that ${\rm Sk}(X^{i})$ collapses to ${\rm Sk}(X^{i+1})$. At the end we obtain that
 ${\rm Sk}(X^0)$ collapses to ${\rm Sk}(X^{m})={\rm Sk}(X^{\circ})$ hence
${\rm Sk}(X^{\circ})$ is a strong deformation retract of $X^{\rm an}$.
\end{proof}

\def\cprime{$'$} \def\cprime{$'$} \def\cprime{$'$} \def\cprime{$'$}
  \def\cprime{$'$} \def\cprime{$'$} \def\dbar{\leavevmode\hbox to
  0pt{\hskip.2ex \accent"16\hss}d} \def\cprime{$'$} \def\cprime{$'$}
  \def\polhk#1{\setbox0=\hbox{#1}{\ooalign{\hidewidth
  \lower1.5ex\hbox{`}\hidewidth\crcr\unhbox0}}} \def\cprime{$'$}
  \def\cprime{$'$} \def\cprime{$'$} \def\cprime{$'$}
  \def\polhk#1{\setbox0=\hbox{#1}{\ooalign{\hidewidth
  \lower1.5ex\hbox{`}\hidewidth\crcr\unhbox0}}} \def\cdprime{$''$}
  \def\cprime{$'$} \def\cprime{$'$} \def\cprime{$'$} \def\cprime{$'$}
\providecommand{\bysame}{\leavevmode\hbox to3em{\hrulefill}\thinspace}
\providecommand{\MR}{\relax\ifhmode\unskip\space\fi MR }
\providecommand{\MRhref}[2]{%
  \href{http://www.ams.org/mathscinet-getitem?mr=#1}{#2}
}
\providecommand{\href}[2]{#2}

\bigskip

\noindent  JK: Princeton University, Princeton NJ 08544-1000

{\begin{verbatim} kollar@math.princeton.edu\end{verbatim}}
\medskip

\noindent  JN: KU Leuven, Dept.\ of Math. Celestijnenlaan 200B, 3001 Heverlee,
Belgium

{\begin{verbatim} johannes.nicaise@wis.kuleuven.be\end{verbatim}}
\medskip

\noindent  CX: Beijing International Center of Mathematics Research, 
Beijing, 100871, China

{\begin{verbatim} cyxu@math.pku.edu.cn\end{verbatim}}

\end{document}